\documentclass[12pt]{article}
\usepackage[a4paper,margin=1in]{geometry}
\usepackage{amsmath,amssymb,amsthm,mathtools}
\usepackage{graphicx}
\usepackage{tikz}
\usepackage{booktabs}
\usepackage{array}
\usepackage{algorithm}
\usepackage{algpseudocode}
\usepackage{hyperref}
 \usepackage{caption}
\usepackage{authblk}
\everymath{\displaystyle}
\usepackage{enumitem}
\renewenvironment{proof}{{\bf \noindent Proof.}}{\qed}
\usepackage{tikz}
\usetikzlibrary{fit,backgrounds}
\usetikzlibrary{arrows.meta,positioning,calc,shapes.geometric}
\usetikzlibrary{fit, backgrounds, shadows}

\hypersetup{colorlinks=true,linkcolor=blue,citecolor=blue,urlcolor=blue}

\newtheorem{theorem}{Theorem}[section]
\newtheorem{lemma}[theorem]{Lemma}
\newtheorem{proposition}[theorem]{Proposition}
\newtheorem{corollary}[theorem]{Corollary}

\theoremstyle{definition}

\theoremstyle{remark}
\newtheorem{remark}[theorem]{Remark}

 \title{On embeddings of the difference graph of the intersection power graph and the power graph}

 \author{Manisha, Ekta, Jitender Kumar$^*$\\
	 \small Department of Mathematics, Birla Institute of Technology and Science Pilani, Pilani-333031, India 
  
  \texttt{yadavmanisha2611@gmail.com, ektasangwan0@gmail.com, jitenderarora09@gmail.com}
  }

\newcommand{\transtitle}[1]{}
\date{}

\begin{document}
 \maketitle


 \begin{abstract}
 The power graph of a finite group $G$ is a simple undirected graph with vertex set $G$ and two vertices are adjacent if one is a power of the other. The intersection power graph of a finite group $G$ is a simple undirected graph with vertex set $G$ and two vertices $x$, $y$ are adjacent if $\langle x\rangle \cap \langle y \rangle \neq \{e\}$. The difference graph $\mathcal{D}(G)$ of a finite group $G$ is the difference of the intersection power graph and power graph with all isolated vertices removed. We characterized all the finite nilpotent groups $G$ except $2$-group such that the difference graph is planar. Further, we determine all the finite nilpotent groups
 whose difference graph has genus at most $2$. Moreover, we prove that there does not exist any finite group whose difference graph is projective planar.

 \end{abstract}

* Corresponding author

\section{Introduction and Preliminaries}
Topological indices are important numerical invariants that capture structural properties of algebraic graphs. They assign a numerical value to a graph based on properties such as vertex degrees, distances between vertices, adjacency relations, connectivity, and other structural features. Moreover, topological indices provide a convenient means of describing and comparing graphs, particularly when their underlying structures are complex. Recently, algebraic graphs on groups, namely the power graph, enhanced power graph, and commuting graph, have been studied extensively due to their valuable applications. The use of algebraic techniques can give new constructions to topological indices. The \emph{power graph} $\mathcal{P}(G)$ of a group $G$ is the simple undirected graph with vertex set $G$, and two distinct vertices $x,y$ are adjacent if one is the power of another, or equivalently: either $x \in \langle y\rangle$ or $y \in \langle x \rangle $. The directed power graph was introduced by Kelarev and Quinn \cite{powerintroduced}. Later on, the undirected power graph was studied by various researchers from different perspectives, see \cite{powerisomorphism,embeddingpower,genus2power}. For more details on the power graph, one can refer to the survey paper \cite{a.kumar2021}. The \emph{intersection power graph} $\mathcal{G}_{I}(G)$ is the graph with vertex set $G$ and two non-identity elements $x,y$ are adjacent if there exists a non-identity element $z$ which is power of both $x$ and $y$, or equivalently: $\langle x\rangle \cap \langle y \rangle \neq \{e\}$ and the identity element is considered to be adjacent to all elements. Intersection power graph was studied by researchers in \cite{intersectionintroduced,li2021finite} and references therein. In the survey \cite{graphdefined} Cameron introduced various difference graphs and open questions. The \emph{difference graph} $\mathcal{D}(G)$ of a finite group $G$ is the difference of intersection power graph and power graph with all the isolated vertices removed. Topological graph theory mainly focuses on the representation of graphs on surfaces so that no two edges cross each other. We say a graph is \emph{embeddable} on a topological surface if it can be drawn on a surface so that no two edges cross each other. A graph is said to be \emph{planar} if it can be drawn on a plane without edge crossing. The \emph{genus} $\gamma(\Gamma)$ of a graph $\Gamma$ is the minimum non-negative integer $g$ such that the graph can be embedded in a sphere with $g$ handles. The problem of calculating the genus is NP-hard (cf. \cite{a.Thomassen1989}). The \emph{cross-cap} $\overline \gamma(\Gamma)$ of a graph $\Gamma $ is the minimum non-negative integer $k$ such that the graph $\Gamma$ can be embedded on a non-orientable surface with $k$-cross-caps.  The problem of embedding graphs on surfaces other than the sphere dates back to the color conjecture due to Heawood \cite{heawood1890map}. It is used in electronic circuit design, where the main purpose is to embed a circuit on a circuit board without any two connections crossing each other. The genus and cross-cap of the graphs defined on groups have been studied extensively. 
The genus and cross-cap of the intersection power graph, the power graph has been studied in \cite{li2021finite, genus2power}. Moreover, the genus and cross-cap of the difference of the enhanced power graph and the power graph have been obtained in \cite{parveen2025nilpotent}. \\
We recall essential definitions and results. 
Let $G$ be a group with the identity element $e$. The order $o(x)$ of an element $x\in G$ is the smallest positive integer $n$ satisfying  $x^n=e$. The least common multiple of the orders of all the elements of $G$  is called the \emph{exponent} of $G$, and it is denoted by $\mathrm{exp}(G)$.  A finite group $G$ is called a \emph{p-group} if $|G|=p^{\alpha}$ for some prime $p$ and integer $\alpha$.  By $\langle x, y\rangle$, we mean the subgroup of $G$ generated by $x$ and $y$. The \emph{normalizer} of a subgroup $H$ in a group $G$ denoted as $N_{G}(H)$  is the largest normal subgroup of $G$ in which $H$ is a normal subgroup and the \emph{centralizer} of a subgroup $H$ in a group $G$ denoted as $C_G(H)$ is the set of all elements in $G$ that commute with every element in $H$. The following results are useful in the sequel.
\begin{theorem}[{\rm \cite[p. 193]{b.dummit1991abstract}}]{\label{nilpotent}} Let $G$ be a finite group. Then the following statements are equivalent:
\begin{enumerate}
    \item[\rm(i)] G is a nilpotent group.
    \item[\rm(ii)] Every Sylow subgroup of $G$ is normal.
    \item[\rm(iii)] $G$ is direct product of its Sylow subgroups.
    \item[\rm(iv)] For $x,y \in G$, $x$ and $y$ commute whenever $o(x)$ and $o(y)$ are relatively prime.
\end{enumerate}
\end{theorem}
\begin{theorem}[{\rm \cite[p. 134]{b.dummit1991abstract}}]
    For any subgroup $H$ of a group $G$, the quotient group $\frac{N_G(H)}{C_G(H)}$ is isomorphic to a subgroup of $Aut(H)$. In particular, $\frac{G}{Z(G)}$ is isomorphic to a subgroup of $Aut(G)$.
\end{theorem}

Now we recall the necessary graph theoretic definitions and notions from  \cite{b.westgraph}. A \emph{graph} $\Gamma$ is a structure $(V(\Gamma), E(\Gamma))$, where $V(\Gamma)$ is the vertex set of $\Gamma$ and $E(\Gamma)\subseteq V(\Gamma)\times V(\Gamma)$ is the edge set of $\Gamma$. If $\{u_1,u_2\}\in E(\Gamma)$, then we say that $u_1$ is \emph{adjacent} to $u_2$ and we denote it by $u_1 \sim u_2$. Otherwise, we write it as $u_1 \nsim u_2$. An edge $\{u,v\}$ in a graph $\Gamma$ is said to be a \emph{loop} if $u=v$. A graph $\Gamma$ is called a \emph{simple} graph if it has no loops or multiple edges. Throughout the paper, we consider only simple graphs. A graph $\Gamma$ is called \emph{cycle} if it is a closed path whose first and last vertex are same.
A \emph{subgraph} $\Gamma'$ of a graph $\Gamma$ is defined as a graph where $V(\Gamma')$ and $E(\Gamma')$ are subsets of  $V(\Gamma)$ and $E(\Gamma)$, respectively. A graph $\Gamma$ is called \emph{complete} if any two vertices of $\Gamma$ are adjacent. The complete graph on $n$ vertices is denoted by $K_n$. A graph $\Gamma$ is called \emph{bipartite} if  $V(\Gamma)$ can be partitioned into two subsets such that no two vertices in the same subset are adjacent.  A \emph{complete bipartite} graph is a bipartite graph having its part sizes $n_1$ and $n_2$ such that every vertex in one part is adjacent to each vertex of the second part, and it is denoted by $K_{n_1,n_2}$. In a graph $\Gamma$, the \emph{subdivision} of an edge $\{u,v\}$ is the operation of replacing $\{u,v\}$ with a path $u\sim w \sim v$ through a new vertex $w$. A  \emph{subdivision} of $\Gamma$ is a graph obtained from $\Gamma$ by successive edge subdivisions. Let $\Gamma_1,\ldots , \Gamma_m$ be $m$ graphs such that $V(\Gamma_i)\cap V(\Gamma_j)= \varnothing$, for distinct $i, j$. Then $\Gamma =\Gamma_1 \cup \cdots \cup \Gamma_m$ is a graph with vertex set  $V(\Gamma_1) \cup \cdots \cup V(\Gamma_m)$ and edge set $E(\Gamma_1) \cup \cdots \cup E(\Gamma_m)$. If $\gamma(\Gamma)$ (or $\overline{\gamma}(\Gamma)$) $=0$, then $\Gamma$ is planar.  A planar graph is said to be \emph{outerplanar} if it can be drawn in the plane such that all its vertices lie on the outer face. A graph with genus $1$ is called \emph{toroidal graph} and a graph with cross-cap $1$ is called \emph{projective planar graph}. The following results are useful for determining the genus and cross-cap of a graph. \begin{theorem}{\rm \cite{b.westgraph}}{\label{planarcondition1}}
A graph $\Gamma$ is planar if and only if it does not contain a subdivision of $K_5$ or $K_{3,3}$.
 \end{theorem}
  \begin{theorem}{\rm \cite{b.westgraph}}{\label{planarcondition1}}
A graph $\Gamma$ is outerplanar if and only if it does not contain a subdivision of $K_4$ or $K_{2,3}$.
 \end{theorem}
\begin{theorem}{\rm \cite[p. 58, p. 152]{b.White1984}} {\label{genuscondition}} 
The genus and cross-cap of the complete graphs $K_n$ and $K_{m, n}$ are given below:
\begin{itemize}
\item[\rm(i)] $\gamma(K_n)= \left\lceil{\frac{(n-3)(n-4)}{12}}\right\rceil $, $n\geq 3$.
\item[\rm(ii)]$\gamma(K_{m,n})=\left\lceil \frac{(m-2)(n-2)}{4}\right\rceil $, $m,n\geq 2$.
\item[\rm(iii)] $\overline{\gamma}(K_n)= \left\lceil{\frac{(n-3)(n-4)}{6}}\right\rceil $, $n\geq 3$, $n\neq 7$; $\overline{\gamma} (K_7)=3$.
\item[\rm(iv)] $\overline{\gamma}(K_{m,n})=\left\lceil \frac{(m-2)(n-2)}{2}\right\rceil $, $m,n\geq 2$.
\end{itemize}
\end{theorem}

\begin{theorem}{\rm \cite{b.White1984}}{\label{m_n_g_formula}}
    Let $\Gamma$ be a simple connected graph with $n$ vertices and $m$ edges, where $n\geq 3$. Then $\gamma(\Gamma)\geq \frac{m}{6} - \frac{n}{2}+1$. Furthermore, equality holds if and only if $\Gamma$ has a triangular embedding.
\end{theorem}

    

\begin{lemma}{\rm \cite{b.Mohar2001graphs}}{\label{genus formula}}
   Let $\phi : \Gamma \rightarrow \mathbb{S}_g$ be a $2$-cell embedding of a connected graph $\Gamma$ to the orientable surface $\mathbb{S}_g$. Then $v-e+f=2-2g$, where $v$, $e$ and $f$ are number of vertices, edges and faces of $\phi (\Gamma)$ respectively, and $g$ is a genus of $\mathbb{S}_g$.
\end{lemma}
The following results of power graphs and difference graphs are useful in the sequel. 
\begin{remark}\label{powergraph}
If $x$ and $y$ are elements of a finite group $G$ such that neither $o(x) \mid o(y)$ nor $o(y) \mid o(x)$, then $x \nsim y$ in $\mathcal{P}(G)$. The converse holds if $x$ and $y$ lie in the same cyclic subgroup.
\end{remark}
\begin{proposition}{\rm \cite[Proposition 2.3]{bera2026difference}}\label{prime}
An element of a group $G$ whose order is prime is isolated in the difference graph.
\end{proposition}

\begin{lemma}{\rm \cite[Lemma 3.6]{bera2026difference}}\label{generators}
Let $G$ be a group. Then for each non-identity element $a \in G$, $\mathcal{G}_a$ is an empty subgraph of $\mathcal{D}(G)$, where $\mathcal{G}_a$ is the set of all generators of the cyclic group generated by the element $a.$
\end{lemma}

\begin{theorem}
    
{\rm \cite[Theorem 4.5]{bera2026difference}}\label{empty}
Let $G$ be a cyclic group of order $n$. Then $\mathcal{D}(G)$ is empty if and only if $\pi(G)\leq 1$, or $G \cong \mathbb{Z}_{pq}$ for some distinct primes $p$ and $q$. In all other cases, $\mathcal{D}(G)$ is connected and diam($\mathcal{D}(G)) \leq 3.$ Moreover, this bound is sharp.
\end{theorem}

\section{Main Results}

\begin{lemma}\label{intersection} 
 Let $G = P_1 \times P_2$ be a finite nilpotent group such that exp$(P_1)=2$ and exp$(P_2)=3$. If the intersection of at most three cyclic subgroups of order $6$ is non-trivial, then  $G$ is isomorphic to either $ \mathbb{Z}_2 \times \mathbb{Z}_2 \times \mathbb{Z}_3 $ or  $\mathbb{Z}_6$.
\end{lemma}
\begin{proof}
 Clearly, $P_1 \cong \underbrace{\mathbb{Z}_2 \times \cdots \times \mathbb{Z}_2}_{{k}\textup{-times}}$. If $k \geq3$, then there are at least $7$ cyclic subgroups of order $6$ in $G$ with non-trivial intersection, a contradiction. If $|P_2| > 3$, then $G$ has at least $4$ cyclic subgroups of order $6$ with non-trivial intersection, again a contradiction. Therefore, $G \cong \mathbb{Z}_2 \times \mathbb{Z}_2 \times \mathbb{Z}_3 $ or $\mathbb{Z}_6$.  
\end{proof}
\begin{theorem} \label{planarity}
Let $G = P_1 \times \cdots \times P_k$ be a finite nilpotent group which is not a $2$-group. Then the difference graph $\mathcal{D}(G)$ is planar if and only if $G$ is isomorphic to one of the following groups: \[\mathbb{Z}_2 \times \mathbb{Z}_2 \times \mathbb{Z}_3,\text{ } \mathbb{Z}_{9} \times \mathbb{Z}_2, \text{ } \mathbb{Z}_4 \times \mathbb{Z}_{p_2},\text{ } p_2 \neq 2.\]
\end{theorem}
\begin{proof} 
Assume that $\mathcal{D}(G)$ is planar. Let $G = P_1 \times \cdots \times P_k$ be a finite nilpotent group such that $|P_i| = p_i^{\alpha_i}$. Let \(k \geq 3\). Consider elements \(x_i\) and \(y_j\) with \(i, j \in \{1, 2, 3\}\) such that $o(x_i)$ is \(p_1 p_3\) and  \(o(y_j)\) is \(p_2 p_3\), where $p_1 < p_2 < p_3$ and $\langle x_i \rangle \cap \langle y_j \rangle \neq \{e\}.$ Furthermore, neither $o(x_i) \mid o(y_j)$ nor $o(y_j) \mid o(x_i)$, which implies that $x_i \sim y_j$ in $\mathcal{D}(G)$. Thus, $\mathcal{D}(G)$ contains $K_{3,3}$ as a subgraph, a contradiction. Therefore, $k\leq2$. Assume that $k=1$ and  so $|G|=p^{\alpha}$. Clearly, $G$ is a non-cyclic $p$-group. Note that there exist two distinct cyclic subgroups with non-trivial intersection. Otherwise, $\mathcal{D}(G)$ is an empty graph. Let $H$ and $K$ be two cyclic subgroups with $o(H)= p^a$ and $o(K)=p^b$, where $a,b \geq 2$. If $p\geq 3$, then both $H$ and $K$ contains at least 6 elements of order $p^a$ and $p^b$, respectively. Therefore, we get $K_{6,6}$ as a subgraph in $\mathcal{D}(G)$, a contradiction. Therefore, $k=2$, which implies that $G \cong P_1 \times P_2.$ \\
\textbf{Case-1:} \emph{$G$ has elements of order $p_1^ip_2$ and $p_1p_2^j$, where $i,\; j \geq 2$.} 
Since $G$ is nilpotent, there exists an element $z \in G$ such that $o(z)=p_1^ip_2^j$. For the subgroup $\langle z \rangle$, note that each element of order $p_1^{i}p_2$ is adjacent to all the elements of order $p_1p_2^{j}$ in $\mathcal{D}(G)$. Therefore, we obtain $K_{3,3}$ as a subgraph in $\mathcal{D}(G)$, a contradiction.\\
\textbf{Case-2:} \emph{$G$ has elements of order $p_1^ip_2$ but not of order  $p_1p_2^j$, where $i,\; j \geq 2$.}\\
\textbf{Subcase-2.1:} \emph{$G$ has at least two cyclic subgroups of order $p_1^ip_2$}. Let if possible, their intersection is non-trivial. There are at least three elements of order $p_1^ip_2$ in each cyclic subgroup. By Remark \ref{powergraph}, we obtain $K_{3,3}$ as a subgraph of $\mathcal{D}(G)$, which is not possible. If $i\geq3,$ then consider the cyclic subgroup of order $p_1^3p_2$ which contains at least $4$ elements of order $p_1^3$ and at least $4$ elements of order $p_1^2p_2$. Moreover, through these elements, we obtain $K_{4,4}$ as a subgraph of $\mathcal{D}(G)$, a contradiction. Now assume that both $p_1,\;p_2 \geq3$. Consider elements $x_i(1 \leq i \leq6)$ and $y_j(1 \leq j \leq8)$ in the cyclic subgroup $\langle z \rangle$ of order $p_1^ip_2$ such that $o(x_i)=p_1p_2$ and  $o(y_j)=p_1^2$ . Notice that $\langle x_i \rangle \cap \langle y_j \rangle \neq \{e\}$. Therefore, we get $K_{6,8}$ as a subgraph in $\mathcal{D}(G)$, a contradiction. We may now assume that either $p_1 =2$, $p_2\geq 3$ or $p_1 \geq 3$, $p_2=2$. Now, if $p_1=2$ and $p_2\geq3$. Clearly the elements of order $2$, $p_2$ and $4p_2$ does not belong to $V(\mathcal{D}(G))$. Therefore, we get $\mathcal{D}(G)$ as disjoint union of $K_{2, \phi(p_2)}$. Notice that such $G$ does not exist. Otherwise, there exists an element $z$ of order $4$ and two elements $x$, $y$ in distinct cyclic subgroups such that $o(x)=o(y)= p_2.$ Consequently, we have two cyclic subgroups $\left<z\right> \times \left<x\right>$ and $\left<z\right> \times \left<y\right>$ of order $4p_2$ having non-trivial intersection, which is not possible. Thus, we have $p_1 \geq 3$ and  $p_2=2$. If $p_1\geq5$ and $p_2=2$, then each element of order $p_1p_2$ is adjacent to all the elements of order $p_1^2.$ Consequently, we get $K_{4,20}$ as a subgraph in $\mathcal{D}(G)$, a contradiction. Therefore, $p_1=3$ and $p_2=2$. Note that a cyclic subgroup of order $p_1^2p_2= 3^22$ has six elements of order $9$ and two elements of order $6$. Clearly, the elements of order $2,3$ and $18$ does not belong to $V(\mathcal{D}(G))$. Therefore, we get $\mathcal{D}(G)$ as disjoint union of $K_{2,6}$. Notice that such $G$ does not exist. Otherwise, there exist two elements $x$, $y$ in distinct cyclic subgroups such that $o(x)=o(y)= 9$. Consequenty, we have two cyclic subgroups $\mathbb{Z}_2 \times \left<x\right> $ and $\mathbb{Z}_2 \times \left<y\right> $ of order $2.3^2$ having non-trivial intersection, which is not possible.\\
\textbf{Subcase-2.2:} \emph{$G$ has precisely one cyclic subgroup of order $p_1^ip_2$, where $i\geq 2$.} By the similar argument used in \textbf{Subcase-2.1}, we have $i=2$ such that either $p_1=2,\; p_2\geq 3$ or $p_1=3, \; p_2=2.$ For $p_1=2, \; p_2\geq3$, similar to the \textbf{Subcase-2.1}, we have $\mathcal{D}(G)\cong K_{2,\phi(p_2)}.$ We claim that $G\cong \mathbb{Z}_4 \times \mathbb{Z}_{p_2}$. Now for $G\cong P_1\times P_2$, note that the order of non-trivial elements of the group $P_2$ is $p_2$. Otherwise, $G$ has an element of order $p_1p_2^j,$ where $j\geq2.$ If $|P_2|\geq p_2^2$, then consider $x,y\in P_2$ such that $x \notin \left< y \right>$  and $o(x)= o(y)=p_2.$ Now, let $z\in P_1,$ such that $o(z)=4.$ Consequently, we get two cyclic subgroups $\left< z\right> \times \left< x\right>$ and $\left<z \right> \times \left< y \right>$ of order $4p_2$, which is not possible. Therefore, $P_2\cong \mathbb{Z}_{p_2}$. Note that $P_1$ has an element of order $4$. We show that $|P_1| \neq 2^{\alpha},$ where $\alpha >2.$ It follows that, $P_1 \cong \mathbb{Z}_4.$ Suppose $|P_1|= 2^{\alpha},$ with $\alpha \geq 3$. If there exist $x,y\in P_1$ such that $o(x)= o(y)=4$ and $x \notin \left< y \right>$, then for $z\in P_2$ such that $o(z)=p_2,$ we obtain two subgroups $\left< x\right> \times \left< z\right>$ and $\left<y \right> \times \left< z \right>$ of order $4p_2$, which is not possible. Thus, there are at least $5$ elements  $x_i(1 \leq i \leq 5)$ of order $2$. Consequently, we have five cyclic subgroups $\left< x_i\right> \times P_2$ of $G$ whose intersection is non-trivial. Consequently, we get $K_5$ as a subgraph of $\mathcal{D}(G)$, a contradiction. If $p_1=3,\; p_2=2$, then by argument used in \textbf{Subcase-2.1}, we have $\mathcal{D}(G)\cong K_{6,2}.$ Then in the similar lines for $p_1=2,\; p_2\geq 3$, one can prove that $G\cong \mathbb{Z}_9 \times \mathbb{Z}_2$.\\
\textbf{Case-3:} \emph{$G$ does not have elements of order $p_1^ip_2$ and $p_1p_2^j$, where $i, j \geq2.$} Then, $G$ has the non-trivial elements of orders  $p_1,\; p_2$ and $p_1p_2$ only. Clearly, $G$ does not have exactly one cyclic subgroup of order $p_1p_2$. Otherwise, $G\cong \mathbb{Z}_{p_1p_2}$ and so by Theorem \ref{empty}, $\mathcal{D}(G)$ is an empty graph. Thus, there exist at least two cyclic subgroups of order $p_1p_2.$ Moreover, their intersection is non-trivial. Otherwise, $\mathcal{D}(G)$ becomes an empty graph.\\
\textbf{Subcase-3.1:} \emph{ Either both $p_1,\; p_2 \geq 3$ or $p_1=2,\; p_2\geq5$ }. Then there exist two cyclic subgroups $H$ and $K$ of order $p_1p_2$ such that $H\cap K \neq \{e\}$. Moreover, each $H$ and $K$ contains at least $4$ elements of order $p_1p_2$. Consequently, we obtain $K_{4,4}$ as a subgraph of $\mathcal{D}(G)$, a contradiction.\\
\textbf{Subcase-3.2:} \emph{$p_1=2$ and $p_2=3.$} If the intersection of at least four cyclic subgroups of order $6$ is non-trivial, then the number of elements of order $6$ in each cyclic subgroup is two. Therefore, we obtain $K_{4,4}$ as a subgraph in $\mathcal{D}(G)$, again a contradiction. Therefore, there are at most three cyclic subgroups of order $6$ with non-trivial intersection. Consequently, by Lemma \ref{intersection}, we have, $G \cong \mathbb{Z}_2 \times \mathbb{Z}_2 \times \mathbb{Z}_3.$ 
Conversely, the planar drawings of these graphs are shown in Figure $1$.
\end{proof}

\begin{figure}[h]
\centering
\includegraphics[width=0.3\textwidth]{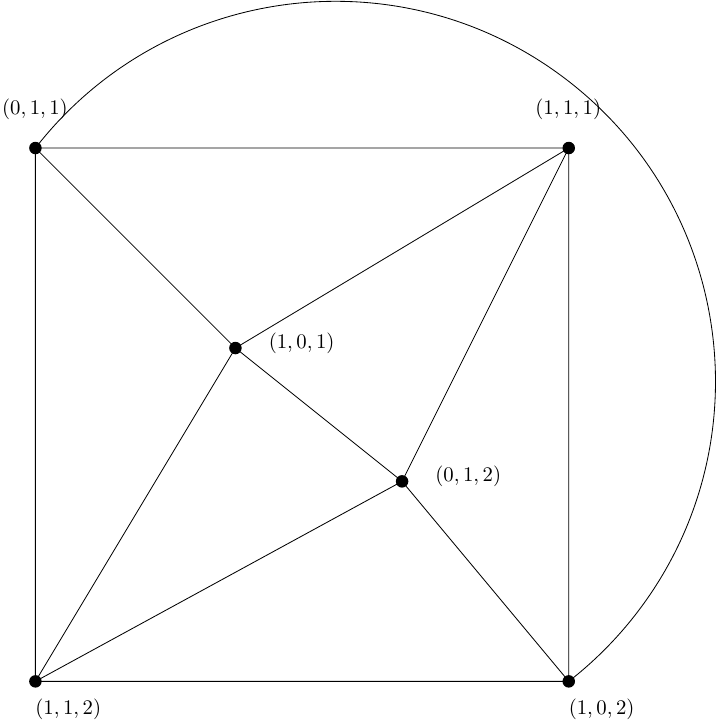}
\hfill
\includegraphics[width=0.3\textwidth]{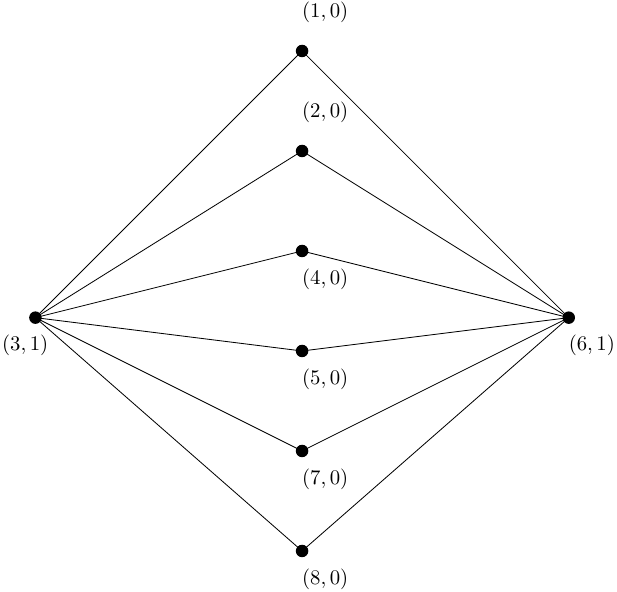}
\hfill
\includegraphics[width=0.38\textwidth]{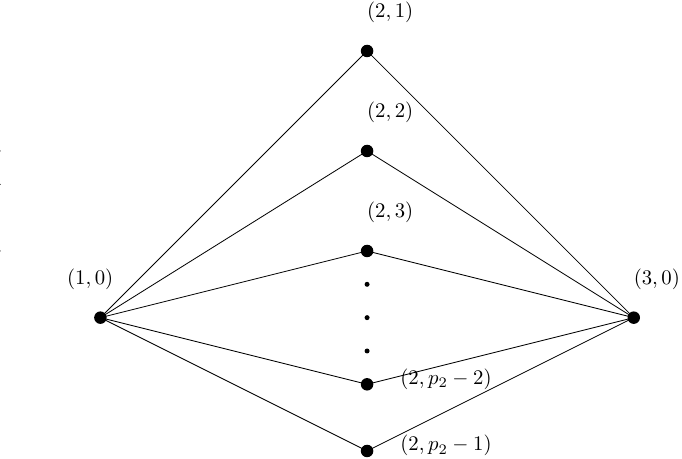}
\caption{Planar drawings of $\mathcal{D}(\mathbb{Z}_{2} \times \mathbb{Z}_{2} \times \mathbb{Z}_{3})$, $\mathcal{D}(\mathbb{Z}_{9} \times \mathbb{Z}_{2})$ and $\mathcal{D}(\mathbb{Z}_{4} \times \mathbb{Z}_{p_2})$}
\label{fig:planar}
\end{figure}
\begin{corollary}
 Let $G \cong P_1 \times \cdots \times P_k$ be a finite nilpotent group which is not a $2$-group. Then the difference graph $\mathcal{D}(G)$ is outerplanar if and only if $G \cong \mathbb{Z}_{4} \times \mathbb{Z}_{3}.$
\end{corollary}

\begin{proof} We first assume that $\mathcal{D}(G)$ is an outerplanar graph. Since every outerplanar graph is a planar graph, it follows from Theorem \ref{planarity} that, if $G\cong \mathbb{Z}_2 \times \mathbb{Z}_2 \times \mathbb{Z}_3$ or  $G\cong \mathbb{Z}_{9} \times \mathbb{Z}_2$. Then, from Figure~$1$, the subgraph induced by the vertices
$\{(0,1,1), (0,1,2), (1,0,1), (1,0,2),(1,1,1)\}$  and $\{(3,1), (6,1), (1,0), (2,0),(4,0)\}$ respectively, is equal to $K_{2,3}$, a contradiction. Now, consider $G\cong\mathbb{Z}_4 \times \mathbb{Z}_{p}$ (where $p \neq 2$). For $p >3$, the set of elements of order $4$ and $2p$  form a subgraph that is equal to $K_{2,3}$, again a contradiction. Therefore, $G \cong \mathbb{Z}_{4} \times \mathbb{Z}_{3}.$ Conversely, the difference graph $\mathcal{D}(\mathbb{Z}_{4} \times \mathbb{Z}_{3})\cong C_4$, which is outerplanar.
    \end{proof}

\begin{lemma}\label{intersection4}
  
 Let $G = P_1 \times P_2$ be a finite nilpotent group such that exp$(P_1)=2$ and exp$(P_2)=3$. Suppose $G$ has precisely four cyclic subgroups of order $6$ with non-trivial intersection, then \( G \) is isomorphic to either \( \mathbb{Z}_2 \times \mathbb{Z}_3 \times \mathbb{Z}_3 \) or \( \mathbb{Z}_2 \times \mathbb{Z}_2 \times \mathbb{Z}_3 \times \mathbb{Z}_3 \).
\end{lemma}
\begin{proof}
 Clearly, $P_1 \cong \underbrace{\mathbb{Z}_2 \times \cdots \times \mathbb{Z}_2}_{{k}\textup{-times}}$. If $k \geq3$, then there are at least $7$ cyclic subgroups of order $6$ in $G$ with non-trivial intersection, a contradiction. If  $|P_2| > 9$, then $G$ has at least $13$ cyclic subgroups of order $6$ with non-trivial intersection, again a contradiction. Therefore, $G \cong \mathbb{Z}_2 \times \mathbb{Z}_3 \times \mathbb{Z}_3 $ or $\mathbb{Z}_2  \times \mathbb{Z}_2 \times \mathbb{Z}_3 \times \mathbb{Z}_3$. 
 \end{proof}
\begin{theorem}\label{genus}
Let $G \cong P_1 \times \cdots \times P_k$ be a finite nilpotent group which is not a $2$-group. Then
\begin{itemize}
    \item[\rm(i)] $\gamma(\mathcal{D}(G))=1$ if and only if $G \cong \mathbb{Z}_2 \times \mathbb{Z}_3 \times \mathbb{Z}_3$ 
    \item[\rm(ii)] $\gamma(\mathcal{D}(G))=2$ if and only if $G \cong \mathbb{Z}_{8}\times \mathbb{Z}_3$.
\end{itemize}
Moreover, there does not exists any group $G$ such that $\overline{\gamma}(\mathcal{D}(G))=1$.
\end{theorem}

\begin{proof}
 Assume that $\gamma(\mathcal{D}(G)) \leq 2$ and $\overline{\gamma}(\mathcal{D}(G))=1$. Let $k \ge 3.$ Consider elements $x_{i}$ ($1 \leq i \leq 4$) and $y_j$($1 \leq j \leq 8$) such that $o(x_i)$ is $p_1p_3$ and $o(y_j)$ is $ p_2p_3$, where $p_1 < p_2 < p_3$  and $\langle x_i \rangle \cap \langle y_j \rangle \neq \{e\}$. Moreover, neither $o(x_i) \mid o(y_j)$ nor $o(y_j) \mid o(x_i)$, which implies that  $x_i \sim y_j$ in $\mathcal{D}(G)$. Thus, $\mathcal{D}(G)$ contains $K_{4,8}$ as a subgraph, which implies that $\gamma(\mathcal{D}(G)) \geq3$ and $\overline{\gamma}(\mathcal{D}(G)) \geq6$. Therefore, $k\leq2$. Assume that $k=1$ and so $|G|=p^{\alpha}$. Clearly, $G$ is a non-cyclic $p$-group. Note that there exist two distinct cyclic subgroups with non-trivial intersection. Otherwise, $\mathcal{D}(G)$ is an empty graph. Let $H$ and $K$ be two cyclic subgroups with $o(H)= p^a$ and $o(K)=p^b$, where $a,b \geq 2$. If $p\geq 3$, then both $H$ and $K$ contains at least 6 elements of order $p^a$ and $p^b$, respectively. Therefore, we get $K_{6,6}$ as a subgraph in $\mathcal{D}(G)$. Consequently, $\gamma(\mathcal{D}(G)) \geq 3$ and $\overline{\gamma}(\mathcal{D}(G)) \geq 3$. Therefore, $k=2$, and it follows that $G \cong P_1 \times P_2.$\\
 \textbf{Case-1:} \emph{$G$ has  elements of order $p_1^ap_2$ and $p_1p_2^b$, where $a,\; b \geq 2$.}
If either $a$ or $b \geq 3$, then the cyclic subgroup $\langle z \rangle $ of order $p_1^{a}p_2^{b}$ has at least six elements $u_{i}$ of order $p_1^{a}p_2$ and eight elements $v_{j}$ of order $p_1p_2^{b}$ with non-trivial intersection. Therefore, $u_{i} \sim v_{j}$ in $\mathcal{D}(G)$ and so  $\mathcal{D}(G)$ contains  $K_{6,8}$ as a subgraph. Consequently, $\gamma(\mathcal{D}(G)) \geq 3$ and $\overline{\gamma}(\mathcal{D}(G)) \geq 3$. Thus, $a=b =2.$ If either $p_1 >3$ or $p_2>3$, then the cyclic subgroup of order $p_{1}^{2}p_{2}^{2}$ has at least eight elements of order $p_1^2p_2$ and twenty elements of order $p_1p_2^2$ with non-trivial intersection. Therefore, we obtain $K_{8,20}$ as a subgraph in $\mathcal{D}(G)$. Consequently, $\gamma(\mathcal{D}(G)) \geq 3$ and $\overline{\gamma}(\mathcal{D}(G)) \geq 3$. Therefore, $p_1=2$ and $p_2=3$.\\  
\textbf{Subcase-1.1} \emph{ $G$ has at least two cyclic subgroups $H$ and $K$ of order $2^23$ with non-trivial intersection.}
Suppose $L$ is a cyclic subgroup of order $2.3^2$. We have the following possibilities:\\
\textbf{Subcase-1.1(a):}\emph{ $(H\cap K)\cap L \neq \{e_G\}$.} Observe that $\mathcal{D}(G)$ contains $K_{8,6}$ as a subgraph, which implies that $\gamma(\mathcal{D}(G)) \geq 3$ and $\overline{\gamma}(\mathcal{D}(G)) \geq 3$.\\
\textbf{Subcase-1.1(b):}\emph{ $(H\cap K)\cap L= \{e_{G}\}$.} 
Note that $H \cap K$ will contain elements of order three only. Otherwise, choose $x_1 \in H \cap K$ such that $o(x_1)=2$ and $l_1, l_2 \in L$ such that $o(l_1)=2$ and $o(l_2)=9$. Consequently, we get two cyclic subgroups $\langle x_1\rangle \times \langle l_2 \rangle$ and $\langle l_1\rangle \times \langle l_2 \rangle$ of order $2.3^2$ having non-trivial intersection. Consequently, we obtain $K_{6,6}$ as a subgraph of $\mathcal{D}(G)$. Hence, $\gamma(\mathcal{D}(G)) \geq 3$ and $\overline{\gamma}(\mathcal{D}(G)) \geq 3$.

\begin{figure}[h]
  \centering
 \includegraphics[width=0.4\textwidth]{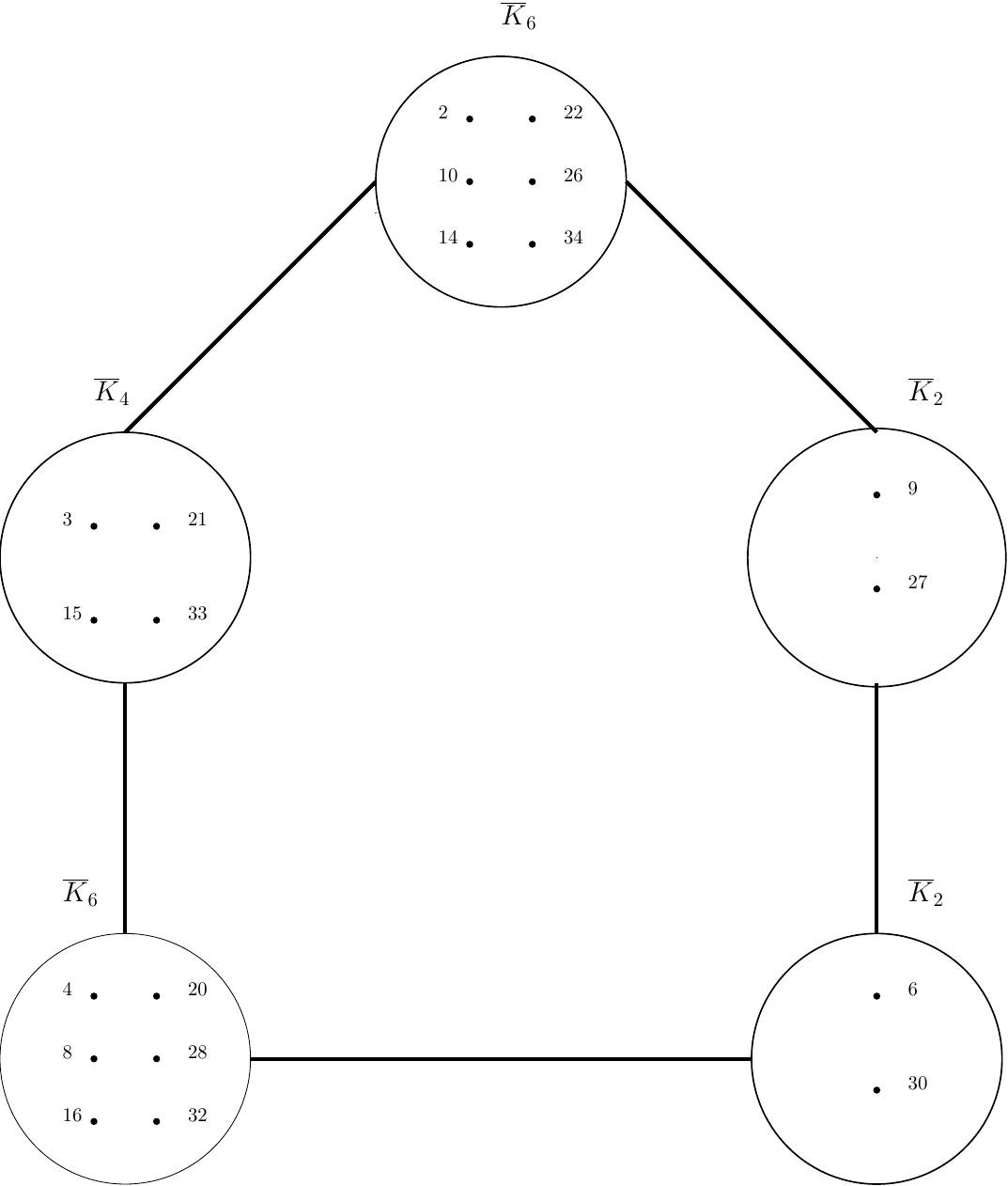}
\caption{  Drawing of $\mathcal{D}(\mathbb{Z}_{36})$}  \label{fig1}
\end{figure}
\textbf{Subcase-1.2:} \emph{$G$ has precisely one cyclic subgroup of order $2^23$ or more than one cyclic subgroups of order $2^2.3$ with trivial intersection.} Notice that $P_1$ has an element of order $4$. We show that $|P_1| \neq 2^{\alpha},$ where $\alpha >2.$ It follows that, $P_1 \cong \mathbb{Z}_4.$ Suppose $|P_1|= 2^{\alpha},$ with $\alpha \geq 3$. If there exist $x,y\in P_1$ such that $o(x)= o(y)=2$, then for $z\in P_2$ such that $o(z)=9,$ we obtain two subgroups $\left< x\right> \times \left< z\right>$ and $\left<y \right> \times \left< z \right>$ of order $2.3^2$ having non-trivial intersection, which is not possible. By the similar argument, one can observe that $|P_2|=9$. Therefore, $P_2$ is isomorphic to either $\mathbb{Z}_3 \times \mathbb{Z}_3$ or $\mathbb{Z}_9$. Since $P_2$ has an element of order $9$, we obtain $G \cong \mathbb{Z}_4 \times \mathbb{Z}_9$. From Figure $2$, we observe that $\mathcal{D}(\mathbb{Z}_{36})$ contains $K_{4,6}$ as a subgraph. Consequently, $\gamma(\mathcal{D}(\mathbb{Z}_{36})) \geq 2$ and $\overline{\gamma}(\mathcal{D}(\mathbb{Z}_{36})) \geq 4$. We claim that $\gamma(\mathcal{D}(\mathbb{Z}_{36})) > 2$. On the contrary, assume that $\gamma(\mathcal{D}(\mathbb{Z}_{36})) = 2$. Then, by Lemma \ref{genus formula}, we obtain $f = 54$. Further, we obtain $2e < 3f$, a contradiction. Therefore, $\gamma(\mathcal{D}(\mathbb{Z}_{36})) > 2$.\\
Other two subcases for the cyclic subgroups of order $2.3^2$ are similar to Subcase $1.1$ and Subcase $1.2$, respectively. Consequently, we get $\gamma(\mathcal{D}(G)) \geq 3$ and $\overline{\gamma}(\mathcal{D}(G)) \geq 3$.

\textbf{Case-2:} \emph{$G$ has elements of order $p_1^ap_2$ but not of order $p_1p_2^b$, where $a,\; b \geq 2$.}\\
 \textbf{Subcase-2.1:} \emph{There exist at least two cyclic subgroups $H$, $K$ of order $p_1^ap_2$ with non-trivial intersection.}
 If possible, let $a \geq 3$. Then there are at least eight elements of order $p_1^ap_2$ in $H$ and $K$, respectively. Consequently, we get $K_{8,8}$ as a subgraph in $\mathcal{D}(G)$ and hence $\gamma(\mathcal{D}(G)) \geq 3$ and $\overline{\gamma}(\mathcal{D}(G)) \geq 3$. Therefore, $a=2$. First, assume that $p_1 \neq 2$ and $p_2 \neq 3$, then we obtain $K_{6,6}$ as a subgraph in $\mathcal{D}(G)$. Consequently, $\gamma(\mathcal{D}(G)) \geq 3$ and $\overline{\gamma}(\mathcal{D}(G)) \geq 3$.  Therefore, we get $p_1=2,\; p_2=3$. Since $H \cap K \neq \{e\}$. Therefore, $H \cap K $ cannot contain elements of orders $3$ and $4$. Otherwise, $H=K= H \cap K$, a contradiction . In all the other possible cases, we obtain $K_{6,6}$ as a subgraph of $\mathcal{D}(G)$. Therefore, $\gamma(\mathcal{D}(G)) \geq 3$ and $\overline{\gamma}(\mathcal{D}(G)) \geq 3$.\\
\textbf{Subcase-2.2:} \emph{$G$ has precisely one cyclic subgroup of order $p_1^ap_2$ or at least two cyclic subgroups of order $p_1^ap_2$ with trivial intersection.}

 \begin{figure}[h]
  \centering
 \includegraphics[width=0.6\textwidth]{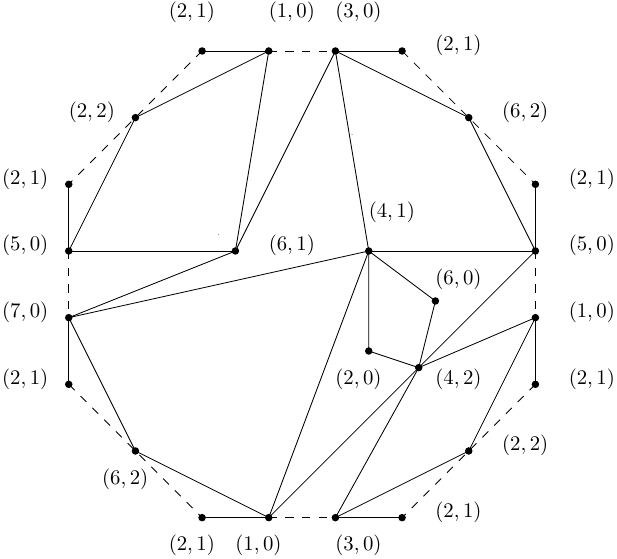}
\caption{Embedding of $\mathcal{D}(\mathbb{Z}_{8} \times \mathbb{Z}_{3} )$ in $\mathbb{S}_2$}    \label{fig2}
\end{figure}
 If $a = 2$ and $p_1, p_2 \geq 3$, then $\mathcal{D}(G)$ contains $K_{6,8}$ as a subgraph and, therefore, $\gamma(\mathcal{D}(G)) \geq 3$ and $\overline{\gamma}(\mathcal{D}(G)) \geq 3$. If $p_1 \geq 5$ and $p_2 = 2$, then the elements of order $p_1^2$ are adjacent to the elements of order $p_1 p_2$. Consequently, we obtain $k_{4,20}$ as a subgraph in $\mathcal{D}(G)$, and so $\gamma(\mathcal{D}(G)) \geq 3$ and $\overline{\gamma}(\mathcal{D}(G)) \geq 3$. In all the remaining possibilities, the difference graph $\mathcal{D}(G)$ is planar (see the argument used in Subcase $2.2$ of Theorem \ref{planarity}). If possible, assume that $a\geq3$. If $p_1\neq 2, p_2\neq 3$, then one can observe that $\mathcal{D}(G)$ contains $K_{4,8}$ as a subgraph, because every element of order $p_1^3$ is adjacent to every element of order $p_1^2p_2$ in $\mathcal{D}(G)$. It follows that $\gamma(\mathcal{D}(G)) \geq 3$ and $\overline{\gamma}(\mathcal{D}(G)) \geq 3$. Therefore, $p_1=2 \; \text{and} \; p_2=3$. Further, if $a>3$, then $\mathcal{D}(G)$ contains $K_{8,8}$ as a subgraph, because every element of order $2^4$ is adjacent to every element of order $2^33$. Consequently, we again obtain $\gamma(\mathcal{D}(G)) \geq 3$ and $\overline{\gamma}(\mathcal{D}(G)) \geq 3$. Thus, we must have $a = 3$ and $p_1 = 2$, $p_2 = 3$. If $|P_2| > 3$. Let $x,y \in P_2$ such that $o(x)= o(y)=3$ and $x \notin \langle y \rangle$. Then for $z \in P_1$ such that $o(z)=8$, we get two cyclic subgroups 
 $ \langle z \rangle \times \langle x \rangle$ and $\langle z \rangle \times \langle y \rangle$ of order $2^3.3$ having a non-trivial intersection, which is a contradiction. Therefore, $P_2 \cong \mathbb{Z}_3$. Note that $P_1$ has a unique cyclic subgroup of order $8$, which is normal in $P_1$. Let $z \in P_1$ such that $o(z)=8$. If there exists an element $x \in P_1 \setminus \langle z \rangle$ such that $x \in {C}_{P_1}( z)$, then $P_1$ has a subgroup isomorphic to $\mathbb{Z}_2 \times \mathbb{Z}_8$, which is not  possible because $\mathbb{Z}_2 \times \mathbb{Z}_8$ has two cyclic subgroups of order $8$. Thus, ${C}_{P_1}( z)= \langle z \rangle.$ This implies that $P_1 / \langle z \rangle$ is isomorphic to a subgroup of the automorphism group of $\langle z \rangle$. Thus, we conclude that $|P_1|= 8, 16, 32$, and so  $P_1 \cong \mathbb{Z}_8, D_{16}, Q_{16}$ or $SD_{16}$. Thus, $G \cong \mathbb{Z}_{8} \times \mathbb{Z}_{3}, Q_{16}\times \mathbb{Z}_3$, $SD_{16}\times \mathbb{Z}_3$ or $D_{16} \times \mathbb{Z}_3$. If $G \cong  Q_{16}\times \mathbb{Z}_3$, $SD_{16}\times \mathbb{Z}_3$ or $D_{16} \times \mathbb{Z}_3$, then $\mathcal{D}(G)$ contains $K_{4,8}$ as a subgraph, which implies that $\gamma(\mathcal{D}(G)) \geq 3$ and $\overline{\gamma}(\mathcal{D}(G))\geq 3$. Therefore, $G\cong \mathbb{Z}_{8} \times \mathbb{Z}_{3}$. On the other hand, genus $2$ embedding of $\mathcal{D}(\mathbb{Z}_{8} \times \mathbb{Z}_{3}) $ is given in Figure $3$. \\
\textbf{Case-3:} \emph{$G$ does not have elements of order $p_1^ap_2$ and $p_1p_2^b$ with $a,\; b \geq 2$.} \\
\textbf{Subcase-3.1:} \emph{$G$ has at least five cyclic subgroups of order $p_1p_2$ with non-trivial intersection.} 
 Since there are at least $2$ elements of order $p_1p_2$ in each cyclic subgroup, we get $K_{4,6}$ as a subgraph in $\mathcal{D}(G)$. Consequently, $\gamma(\mathcal{D}(G)) \geq 2$ and $\overline{\gamma}(\mathcal{D}(G)) \geq 4$. We claim that $\gamma(\mathcal{D}(G)) > 2$. On the contrary, assume that $\gamma(\mathcal{D}(G)) = 2$. Then, by Lemma \ref{genus formula}, we have $f = 28$. Indeed, we get $2e < 3f$, a contradiction. Therefore, $\gamma(\mathcal{D}(G)) \geq 3$.\\
\begin{figure}[h]
  \centering
\includegraphics[width=0.4\textwidth]{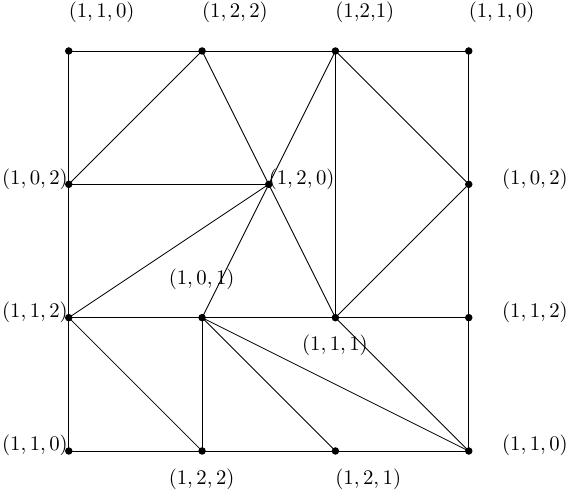}
\caption{ Embedding of $\mathcal{D}(\mathbb{Z}_2 \times \mathbb{Z}_3 \times \mathbb{Z}_3)$ in $\mathbb{S}_1$}  \label{fig3}
\end{figure}\\
\textbf{Subcase-3.2:} \emph{$G$ has precisely four cyclic subgroups of order $p_1p_2$ with non-trivial intersection.} If $p_1p_2\neq 6$, then we get $K_{8,8}$ as a subgraph in $\mathcal{D}(G)$, because there are at least $4$ elements of order $p_1p_2$ in each cyclic subgroup. Therefore, $p_1,p_2\in \{2, 3\}$. Without loss of generality, let $p_1=2$ and $p_2= 3$. By Lemma \ref{intersection4}, we get $G \cong \mathbb{Z}_2 \times \mathbb{Z}_3 \times \mathbb{Z}_3 $ or $G \cong \mathbb{Z}_2 \times \mathbb{Z}_2 \times \mathbb{Z}_3 \times \mathbb{Z}_3 $. From the graph of $\mathcal{D}(\mathbb{Z}_2 \times \mathbb{Z}_3 \times \mathbb{Z}_3)$, observe that $\mathcal{D}(\mathbb{Z}_2 \times \mathbb{Z}_3 \times \mathbb{Z}_3)$ has $K_{4,4}$ as a subgraph. Consequently, $\gamma(\mathcal{D}(\mathbb{Z}_2 \times \mathbb{Z}_3 \times \mathbb{Z}_3)) \geq 1$ and $\overline{\gamma}(\mathcal{D}(\mathbb{Z}_2 \times \mathbb{Z}_3 \times \mathbb{Z}_3)) \geq 2$. The genus $1$ drawing of $\mathcal{D}(\mathbb{Z}_2 \times \mathbb{Z}_3 \times \mathbb{Z}_3)$ is given in Figure~$4$. If $G\cong \mathbb{Z}_2 \times \mathbb{Z}_2 \times \mathbb{Z}_3 \times \mathbb{Z}_3$, then $\gamma(\mathcal{D}(\mathbb{Z}_2 \times \mathbb{Z}_2 \times \mathbb{Z}_3 \times \mathbb{Z}_3)) \geq 1$ and $\overline{\gamma}(\mathcal{D}(\mathbb{Z}_2 \times \mathbb{Z}_2 \times \mathbb{Z}_3 \times \mathbb{Z}_3)) \geq 2$. We claim that $\gamma(\mathbb{Z}_2 \times \mathbb{Z}_2 \times \mathbb{Z}_3 \times \mathbb{Z}_3)) > 2$. On the contrary, let $\gamma(\mathcal{D}(\mathbb{Z}_2 \times \mathbb{Z}_2 \times \mathbb{Z}_3 \times \mathbb{Z}_3)) \in \{1,2\}$. Then, by Lemma \ref{genus formula}, we have $f \in \{96, 94 \}$. Since $|E(\mathcal{D}(\mathbb{Z}_2 \times \mathbb{Z}_2 \times \mathbb{Z}_3 \times \mathbb{Z}_3))|= 120$, we have $2e < 3f$, a contradiction. Therefore, $\gamma(\mathcal{D}(\mathbb{Z}_2 \times \mathbb{Z}_2 \times \mathbb{Z}_3 \times \mathbb{Z}_3)) \geq 3$.

\end{proof}

\section{Conclusion and Future Work}
In this paper, we investigated the difference graph associated with finite nilpotent groups that are not $2$-groups, from the perspective of genus and cross-cap. The genus and cross-cap measure the minimum complexity of a surface required to embed a graph without edge crossings. We classified finite nilpotent groups that are not $2$-groups for which the difference graph $ \mathcal{D}(G)$ is planar, and characterized the group classes in which the genus of $ \mathcal{D}(G)$ is one or two. We also established that there does not exist any finite nilpotent group (other than a $2$-group), whose difference graph has cross-cap one.\\
 The results obtained in this paper open up several interesting directions for future research. One can extend the work considered in this paper from finite nilpotent groups to arbitrary finite groups. Based on our observations, we propose the following conjecture regarding cross-cap two of the difference graph $\mathcal{D}(G)$.\\
\textbf{Conjecture.} Let $G$ be a finite nilpotent group which is not a $2$-group. Then $\overline{\gamma}(D(G))=2 \quad \text{if and only if} \quad G\cong \mathbb{Z}_2 \times \mathbb{Z}_3 \times \mathbb{Z}_3.$\\

\section*{Declarations}

\textbf{Funding}: The first author extends their gratitude to the Birla Institute of Technology and Science (BITS) Pilani, India, for providing financial support. The second author gratefully acknowledges for providing financial support to CSIR(09/0719(17365)/2024-EMR-I), Government of India.

\textbf{Conflicts of interest/Competing interests}: There is no conflict of interest regarding the publishing of this paper.

\textbf{Availability of data and material (data transparency)}: Not applicable.

\vspace{.3cm}
\textbf{Code availability (software application or custom code)}: Not applicable. 
 

\end{document}